\documentclass[11pt]{article}
\usepackage [dvips]{graphics}
\usepackage[centertags]{amsmath}
\usepackage{amsfonts}
\usepackage{amssymb}
\usepackage{amsthm}
\usepackage{color,epsfig}
\usepackage{newlfont}
\usepackage{stmaryrd}

\theoremstyle{plain}
\newtheorem{thm}{Theorem}[section]
\newtheorem{cor}[thm]{Corollary}
\newtheorem{lem}[thm]{Lemma}
\newtheorem{prop}[thm]{Proposition}

\newtheorem{rem}[thm]{Remark}
\newtheorem{defi}[thm]{Definition}

\def\sqr#1#2{{\vcenter{\vbox{\hrule height.#2pt
              \hbox{\vrule width.#2pt height#1pt \kern#1pt \vrule
width.#2pt}
              \hrule height.#2pt}}}}
\def\Sp{{\mathrm {Sp}}}

\def\lb{\label}

\def\d{{d\over dt}}

\def\3n{\negthinspace \negthinspace \negthinspace }
\def\2n{\negthinspace \negthinspace }
\def\1n{\negthinspace }

\def\R{{\mathbb R}}

\def\ga{{\gamma}}

\def\cH{{\cal H}}

\def\cT{{\cal T}}

\def\no{\noindent}

\def\bs{\bigskip}

\def\dim{\hbox{\rm dim$\,$}}

\def\({\Big (}
\def\){\Big )}
\def\[{\Big[}
\def\]{\Big]}

\def\be{\begin{equation}}
\def\bel{\begin{equation}\label}
\def\ee{\end{equation}}
\def\bea{\begin{eqnarray}}
\def\eea{\end{eqnarray}}
\def\bt{\begin{theorem}}
\def\et{\end{theorem}}
\def\bc{\begin{corollary}}
\def\ec{\end{corollary}}
\def\bl{\begin{lemma}}
\def\el{\end{lemma}}
\def\bp{\begin{proposition}}
\def\ep{\end{proposition}}
\def\br{\begin{remark}}
\def\er{\end{remark}}
\def\ba{\begin{array}}
\def\ea{\end{array}}
\def\bd{\begin{definition}}
\def\ed{\end{definition}}

\DeclareMathOperator{\Id}{Id} 
\DeclareMathOperator{\image}{im}

\DeclareMathOperator{\Graph}{Gr}
\DeclareMathOperator{\Dist}{dist}

\makeatletter
   
   \@addtoreset{equation}{section}
\makeatother

\begin{document}

\title{\bf  Decomposition of Spectral flow and Bott-type Iteration Formula  }\author{Xijun Hu\thanks{Partially supported
by NSFC(No. 11425105, 11790271), E-mail:xjhu@sdu.edu.cn } \quad Li Wu \thanks{ 
 E-mail: 201790000005@sdu.edu.cn }    \\ \\
 Department of Mathematics, Shandong University\\
Jinan, Shandong 250100, The People's Republic of China\\
}

\maketitle
\begin{abstract} 
Let $A(t)$ be a continuous path of Fredhom operators,  we first  prove that the spectral flow $sf(A(t))$ is  cogredient invariant. Based on this property,  we give a decomposition formula of spectral flow   if the path is invariant  under a matrix-like  cogredient.  As applications, we give the generalized Bott-type iteration formula  for linear Hamiltonian systems.   \end{abstract}

\bs

\no{\bf 2010 Mathematics Subject Classification}: 58J30, 37B30, 53D12,  

\bs

\no{\bf Key Words}. spectral flow;  cogredient invariant; decomposition formula; Bott-type iteration formula

\section{Introduction}

In this paper, we consider the decomposition of spectral flow for a path of self-adjoint Fredholm operators. Let  $\cH$ be a separable Hilbert space,   and  we denote by   
 $\cal{FS} (\cH)$ be the set of all  densely defined self-adjoint Fredholm operator on $\cH $. We always equipped   $\cal{FS} (\cH)$  with the gap topology. 
  For a continuous path $A(s)\in \mathcal{FS}(\cH)$, $t\in[a,b]$. The spectral flow $sf(A(t); t\in[a,b])$ is an integer  that counts the net number of eigenvalues that change sign. This notation is first introduced by  Atiyah-Patodi-Singer \cite{APS76} in their study of index theory on manifolds with boundary, since then it had found many significant applications, see \cite{ZL99,BLP05}   and reference therein.

Some basic property of spectral flow such as homotopy invariant, path additivity, direct sum e.t.  are well known, please refer the Appendix.  Our first result is   another basic property which is called   cogredient invariant property of spectral flow.   For convenience,  we first  introduce some notations. 
Let $\cH_1, \cH_2$ be separable Hilbert space,  we denote by  $\cal{L}(\cH_1, \cH_2)$ and  $\cal{C}(\cH_1, \cH_2)$ the set of bounded and closed operators from $\cH_1\to \cH_2$.   We  let $\cal{S}(\cH)$ be the set of self-adjoint operators on $\cH$. 
 For convenience, we denote by $\cal{L}^*(\cH_1, \cH_2), \cal{C}^*(\cH_1, \cH_2), \cal{S}^*(\cH), \cal{FS}^* (\cH)$ be the invertible subsets.  
\begin{thm}\label{th:coginva}  Let 
 $M_s\in C([a,b], \cal{L}^*(\cH_1, \cH_2))$,  $A_s\in C([a,b], \mathcal{FS}(\cH_2)) $,   then 
 \bea M_s^*A_sM_s\in C([a,b], \mathcal{FS}(\cH_1)), \label{th1f1}\eea and we have  
 \bea sf(A_s; s\in[a,b])=sf(M_s^*A_sM_s; [a,b]). \label{sfcog}  \eea
\end{thm}
In a  preprint paper \cite{FPS06},  Fitzpatrick-Stuar-Pejsachowicz proved \eqref{sfcog} in the case that $M_s$ is constant, the domain of $A_s$ is  fixed and 
both	$A_a, A_b$  are invertible. Theorem \ref{th:coginva} can be consider as a generalization of their result. 

Our second main result is  the decomposition formula based on the  cogredient invariant property.   Let $\cH_i$ be closed subspace of $\cH$ for $i=1,\cdots,m$, then we define $$ \sum_{1\leq i\leq m} \cH_i=\cH_1+\cdots \cH_m $$
which is the subspace spanned by   $\cH_i$,  $i=1,\cdots,m$. 
  Suppose $g\in\cal{L}(\cH)$, we call $g$ is a matrix-like operator if 
$\sigma(g)=\{\lambda_1,\cdots,\lambda_n\}$ is finite and there exist $m>0$, such that
\bea \cH=\sum_{1\leq i\leq n}\ker(g-\lambda_i)^m.  \eea
We denote by $\mathcal{M}(\cH)$ be the set of matrix-like operators.   For $g\in\cal{M}(\cH)$,  $\lambda\in \sigma(g) $, we set $$\cH_\lambda:=\ker(g-\lambda)^m, $$
and denote \bea F_\lambda=\left\{\begin{array}{cc} \cH_\lambda, \quad  if\quad  \lambda\in \mathbb{U};
\\ \cH_\lambda+\cH_{\bar{\lambda}^{-1}}, \quad \quad  if \quad  \lambda \notin \mathbb{U}.   \end{array}\right.\eea
Then we have 
\bea \cH=\sum_{1\leq i\leq k}F_{\lambda_i}. \nonumber \eea 
Moreover, let $\hat{F}=span\{F_\lambda, \lambda\in\sigma(g)\cap\mathbb{U}^c \} $, we have  the next theorem.
\begin{thm}\label{th:decom} Suppose $g\in \cal{M}(\cH) $ is invertible and preserve the domain of $A$,  $\sigma(g)\cap \mathbb{U}=\{\lambda_1,\cdots,\lambda_j \}$.  Assume    \bea g^* A_sg=A_s,\quad for \quad  s\in[0,1], \label{coinvariant} \eea    then we have 
\bea sf(A_s)=sf(A_s|_{F_{\lambda_1}})+\cdots+sf(A_s|_{F_{\lambda_j}})+\frac{1}{2}(dim\ker (A_1|_{\hat{F}}) -dim\ker (A_0|_{\hat{F}})) . \label{decomf}     \eea 
\end{thm}
   In \cite{HS09}, by assume $g$ is unitary and $\sigma(g)$  is finite,   Hu-Sun proved the decomposition formula 
   \bea  sf(A_s)=sf(A_s|_{\ker(g-\lambda_1})+\cdots+sf(A_s|_{\ker(g-\lambda_j)})  \label{decomfHS}  \eea  under the condition \bea A_sg=gA_s.  \eea
   Obviously, we give a generalization of \eqref{decomfHS}.   In fact, there is a significant difference is we are not assume $g$ is unitary  in Theorem \ref{th:decom}, hence the subspace $\cH_\lambda$ is not orthogonal.  To overcome this difficulty, we develop a new technique (Lemma \ref{abstract_decomposition})  to prove the equality of spectral flow.

As an applications of  Theorem \ref{th:decom}, we give generalization for the Bott-type iteration formula which is a powerful tool in study the multiplicity and stability of periodic orbits  in Hamiltonian systems. 
In 1956, Bott got his celebrated iteration formula for the Morse index of closed geodesics \cite{Bot56}, and it was generalized by \cite{BTZ82, CD77, CZ84, Eke90}. The precise iteration formula of the general Hamiltonian system was established by Long \cite{Lon99,Lon02}. In fact, the iteration could be regarded as a unitary group action.  Motivated by the symmetry orbits in $n$-body problem \cite{FT04},  Hu-Sun \cite{HS09}  use this opinion to give generalization of Bott-type iteration formula to the system under a circle-type symmetry or brake symmetry group action, and  prove  the stability of Figure Eight orbit \cite{CM00}.  The case  of the brake symmetry was deeply studied  in \cite{LZZ06,LZ14a,LZ14b,HPY17b}.

Based on Theorem \ref{th:decom}, we prove  the  Bott-type iteration formula which cover all the previous cases and moreover give some new generalizations. Our generalized formula could be applied to the closed geodesics on Semi-Riemanian manifold and  heteroclinic orbits with brake symmetry. 

Now we consider the linear Hamiltonian system \bea  \dot{x}(t)=JB(t)x(t),\quad  t\in\mathcal{I}, \label{ham1}\eea
where  $J=\begin{pmatrix}
0 & -I_n\\
I_n & 0
\end{pmatrix}$,  $\mathcal{I}\subset \mathbb{R}$ is a connected subinterval,  $B(t)\in C(\mathcal{I}, \mathcal{S}(\mathbb{R}^{2n}))$.  
In the case $\mathcal{I}$ is finite, the  boundary conditions is given by the Lagrangian subspaces. Let $(\mathbb{C}^{2n},\omega)$ be the standard symplectic 
space with $\omega(x,y)=(Jx,y)$.  A  Lagrangian subspace $V$ is a $n$-dimensional subspace with  $\omega|_V=0$. We denote the set of Lagrangian subspace by $Lag(2n)$.    It is obvious $(\mathbb{C}^{2n}\oplus\mathbb{C}^{2n},-\omega\oplus\omega)$ is a $4n$-dimensional symplectic space, then for $\mathcal{I}=[a,b]$, the boundary condition is given by 
\bea (x(a),x(b))\in \Lambda\in Lag(4n). \eea 
In the case $\mathcal{I}=\mathbb{R}$, we always assume $B(\pm\infty)=\lim_{t\to\pm\infty}B(t)$ exist and $JB(\pm\infty)$ is hyperbolic, that is \bea \sigma(JB(\pm\infty))\cap i\R=\emptyset. \eea
Let  $\cH=L^2(\mathcal{I}, \mathbb{C}^{2n})$   and  $E$ is $W^{1,2}(\mathcal{I}, \mathbb{C}^{2n})$  which satisfied some boundary conditions.  We denote by $$
A:= -J\frac{d}{dt} : E \subset \cH\to \cH.
$$
and $B\in\mathcal{\cH}$  be the multiplicity operator of $B(t)$. Let  $A_s=A-sB$ for $s\in\R$, then $A_s\in\mathcal{FS}(\cH)$.
For $g\in\mathcal{\cH}$, $gE=E$  and satisfied \bea  g^*Ag=A, \quad  g^*Bg=B, \eea  then $g^*A_sg=A_s$. By construct  $g$, we get the spectral flow decomposition of $sf(A_s)$.   We list $6$-cases which common in applications of Hamiltonian systems. Our results generalization all the previous results, especially for the brake symmetry of heteroclinic orbits (Case $5$ and $6$), our result is new.  Please see Section 4 for the detail.

It is well known that spectral flow is equal to  Maslov index, and this is also true for the unbounded domain, see \cite{CLM94,ZL99,RS95, CH07, HP17} and reference therein.  The Maslov index is associated integer to a pair of continuous  path $f(t)=(L_1(t), L_2(t))$, $t\in\mathcal{I}$, in $Lag(2n)\times Lag(2n)$ \cite{CLM94}. From the decomposition of spectral flow, we get the decomposition of Mslov index, please refer Section 5 for the detail.  For reader's convenience, we give a brief describe for the Maslov index and spectral flow in the Appendix.

This paper is organized as follows.  We proved Theorem \ref{th:coginva}  in Section 2 and Theorem \ref{th:decom} in Section 3. In Section 4, we list $6$-cases of decompositions in Hamiltonian systems. In Section 5, we give the some case of the Bott-type iteration formulas. At last, we  briefly review the basic property  of spectral flow and Maslov index in the Section 6.

\section{Spectral flow is  preserved under cogredient } 

Let $V$ be a closed space of $\cH$, and
 $P_V$ be the orthogonal projection from $\cH$ to $V$. 
For  $A\in\cal{C}(\cH)$, we denote the operator $P_VAP_V:V\to V$ by $A_V$. Then $A_V\in \cal{C}(V)$.  Obviously, if  $A\in \cal{S}(\cH)$ then $A_V\in \cal{S}(V)$.

\begin{defi}
Let $A:[a,b]\rightarrow \cal{FS} (\cH)$ be a continuous  curve. We call $A(t)$ is a positive curve if
$\{t, \ker A(t)\neq0\}$ is a distinct set and
 \bea sf(A(t); [0,1])=\sum_{a< t\leq b} \dim \ker(A(t)).\eea
 \end{defi}
Let $A\in \cal{FS}(\cH)$ and $B\in \cal{L}(\cH)\cap \cal{S}(\cH)$, 
then  $A+tB\in \cal{FS}(\cH)$ with $t\in \R$.
Note that  it is positive if  $B|_{\ker(A+tB)>0}$ for any
$t\in \{t|\ker(A+tB)\neq 0\}$.
For example, $A+t I$ is a positive curve with $A\in \cal{FS}(H)$ for $t\in\R$.

Let $S\subset\cal{FS}(\cH)$ be a path connected  subset.
We assume there exist $K\in\cal{L}(\cH)$  such that $(Kx,x)>0,\, \forall x\in \cH$ and for any $A\in S$,  there is a neighborhood $U$ of $A$,  and $\epsilon>0$ such that $B+tK\in S, \,  t\in [0,\epsilon]$ for each $B\in U$.
Then $A+tK,\, t\in [0, \epsilon] $ is a positive curve in $\cal{FS}(\cH)$.
Let $\{\cH_k\}, 1\leq k\leq n$ be a family of Hilbert spaces and  $f_k: S\rightarrow \cal{FS}(\cH_k)$ be a family of continuous maps.

We assume that

(a)
For any $ A\in S$,  $f_k(A+tK),t\in [0,\epsilon]$ is a positive path in $\cal{FS}(\cH_k)$.

(b)
For any $A\in S$,   $\sum_{1\leq k\leq n}\dim\ker f_k(A)=\dim\ker A $.
\\
Then we have the following lemma.
\begin{lem}\label{abstract_decomposition}
Let $A \in C([0,1], \mathcal{FS}(\cH))$ and satisfied condition (a) and (b), 
we have
\begin{equation}
  sf(A(t); t\in [0,1])=\sum_{1\leq k\leq n} sf (f_k(A(t)); t\in [0,1]). \label{eq2.2}
\end{equation}
\end{lem}
\begin{proof}
Since the spectral flow satisfied the Path additivity property, 
we only need to prove \eqref{eq2.2} locally.
 Let $h_k(s,t)=f_k(A(t)+s K)$, $t\in [0,1],s\in [0,\epsilon]$, 
 then for any $t\in [0,1]$,  $h_k(s,t)$ is a positive curve with  $1\leq k\leq n$. Let $t_0\in [0,1]$, since $(Kx,x)>0$ for $x\in \ker A(t_0)$, there is $\delta>0$ such that
\begin{equation*}
\dim\ker(A(t_0)+\delta K)=0.
\end{equation*}
It follows that $\dim\ker (h_k(\delta,t_0))=0$ for $1\leq k\leq n$.
Note that $A(t_0)+\delta K$ is a Fredholm operator, 
so there is $\delta_1>0$ such that
\bea \dim \ker( A(t)+\delta K) =0,\,  \forall t\in [t_0-\delta_1,t_0+\delta_1].\nonumber\eea
It follows that $\dim\ker (h_k(\delta,t))=0$ for $t\in[t_0-\delta_1,t_0+\delta_1]$, $1\leq k\leq n$.
Then we have
\begin{equation*}
 \begin{cases}
sf(A(t)+\delta K, t\in [t_0-\delta_1,t_0+\delta_1])=0\\
 sf(h_k(\delta,t), t\in [t_0-\delta_1,t_0+\delta_1])=0
 \end{cases}.
\end{equation*}
By homotopy invariance of spectral flow, we have
 \bea 
 sf(A(t); t\in [t_0-\delta_1,t_0+\delta_1])=sf(A(t_0-\delta_1+sK); s\in [0,\delta])-sf(A(t_0+\delta_1+sK); s\in [0,\delta]) \nonumber \eea
and
\bea
   sf(h_k(0,t); t\in [t_0-\delta_1,t_0+\delta_1])=sf(h_k(s,t_0-\delta_1); s\in [0,\delta])-sf(h_k(s,t_0+\delta_1); s\in [0,\delta]).\nonumber
\eea
Note that $A(t_0\pm\delta_1)+sK$, $h_k(s,t_0\pm\delta_1),1\leq k\leq n$ are positive paths.
It follows that
\bea 
 sf(A(t_0\pm \delta_1)+sK; s\in [0,\delta])=\sum_{0<s\leq \delta }\dim\ker (A(t_0\pm \delta_1)+sK)\nonumber\\
 =\sum_{0<s\leq \delta } \sum_{1\leq k\leq n} \dim\ker (h_k(s,t_0\pm\delta_1)\nonumber)\\
 =\sum_{1\leq k\leq n}sf (h_k(s,t_0\pm\delta_1); s\in [0,\delta]). \nonumber\eea
 This complete the proof.
 \end{proof}
 
 Please note that Lemma \ref{abstract_decomposition} can be consider as a generalization of Direct sum property of spectral flow.

In the next, we will prove the spectral flow is invariant under the cogredient. The next Lemma is contained in \cite{FPS06}, but for reader's convenience, we give details  here.

\begin{lem}
Let	$E$ be the domain of $A\in\cal{FS}(\cH_2)$.
	If $M\in \cal{L}(\cH_1,\cH_2)$  is invertible,   then $M^*AM\in \cal{FS}(\cH_1)$ with domain $M^{-1}(E)$.
\end{lem}
\begin{proof}
Since $A\in\cal{FS}(\cH_2)$ with domain $E$, we have $\dim \ker A,\, \dim (\cH_2/\image A) <+\infty$.
Since $M$ is invertible, we have $\ker(M^*AM)= \ker(AM)=M^{-1}\ker A$.
Then $M^{-1}$ induce an isomorphism from $\ker A$ to $\ker (M^*AM)$.
Note that $\image(M^*A M)=M^*\image(A)$.
Then $M^*$ induce an isomorphism from $\cH_2/\image(A)$ to $\cH_1/\image(M^*AM)$.
So $M^*AM$ is a Fredholm operator.
 
Since $A \in \cal{FS}(\cH_2)$ with domain $E$, we see that for each $x\in M^{-1}$, $(AMx,My)=(Mx,AMy)$ if and only if $y\in M^{-1}E$.
It follows that $(M^*AM)^*=M^*AM$ with domain $M^{-1}E$.  
Then we can conclude that  $M^*AM\in \cal{FS}(\cH_1)$.
\end{proof}

Recall that the gap topology can be induced by the gap distance $\hat{\delta}$.
Let $X$ be a Banach space.
Let $M,N$ be two closed linear  spaces of  $X$.
Denote by $S_M$ the unit sphere of $M$.
Then gap distance is defined as 
\begin{equation}
\hat{\delta}(M,N)=\max \{\delta(M,N),\delta(N,M)\},
\end{equation}  where
\begin{equation*}
\delta\{M,N\}:=\begin{cases}
\sup_{u\in S_M}\Dist(u,N), &\text{if} ~ M\neq \{0\}\\
0, & \text{if}~ M=\{0\}
\end{cases}.
\end{equation*} 

The gap distance have the following properties.
\begin{lem}\label{lm:gap_distance_compare}
	Let $X,Y$ be two Hilbert spaces.
	Let $M,N$ be two closed linear subspaces of $X$.
	Let $P,Q\in \cal{L}^*(X,Y)$. 
	Then 
	$\hat\delta(PM,QN)\le \hat\delta(M,N)\max\{\|P\|,\|Q\|\}+\|P-Q\|\max\{\|P^{-1}\|,\|Q^{-1}\|\}$.
\end{lem}
\begin{proof}
	Without loss of generality, we assume that $M,N\neq \{0\}$, and
let $d_1=\hat\delta(M,N)$.
Let $x\in PM$ with $\|x\|=1$, we choose
 $y\in N$ such that $\|P^{-1}x-y\|=\Dist(P^{-1}x,N)$, 
then we have $\|y\|\leq \|P^{-1}x\|\leq \|P^{-1}\|$.
Note that 
\begin{eqnarray*}
\|x-Qy\|\leq \|x-Py\|+\|Qy-Py\| &&\leq \|P\|\|P^{-1}x-y\|+\|Q-P\|\|y\|\\
 &&\leq \|P\|\Dist(P^{-1}x,N)+\|Q-P\|\|P^{-1}\|\\
 &&\leq \|P\|\delta(M,N)+\|Q-P\|\|P^{-1}\|.
\end{eqnarray*}
It follows that
$\delta(PM,QN)\leq \|P\|\delta(M,N)+\|Q-P\|\|P^{-1}\|$.
Similarly, we have
$\delta(QN,PM)\leq \|Q\|\delta(N,M)+\|Q-P\|\|Q^{-1}\|$.
This conclude the proof.
\end{proof}

\begin{lem}\label{lem: continuous}   Suppose $M_s\in C([0,1], \cal{L}^*(\cH_1,\cH_2))$,  $A_s\in C([0,1],\cal{FS}(\cH_2))$, then  $M_s^*A_sM_s\in C([0,1],\cal{FS}(\cH_1)) $.  \end{lem}
\begin{proof} 
We only need to show that $M_s^*A_sM_s$ is a continuous curve with the gap topology.
Let $E_s$ be the domain of $A_s$. 
Note that
$$
\Graph(M_s^*A_sM_s)=\{(M_s^* A_s x,M_s^{-1}x)|x\in E_s\}.
$$
Let $Q_s:\cH_2\oplus \cH_2 \to \cH_1\oplus \cH_1$ be 
$\begin{pmatrix}
	M_s^* & 0\\
	0& M_s^{-1}\\
\end{pmatrix} $, 
then $Q_s \in  C([0,1], \cal{L}^*(\cH_2\oplus \cH_2,\cH_1\oplus\cH_1))$, and
we also have  $\Graph(M_s^*A_sM_s)=Q_s\Graph(A_s)$.
Since  $\|Q_s\|$ and $\|Q_s^{-1}\|$ are continuous functions on $[0,1]$,
 we have $\|Q_s\|>0$, $\|Q_s^{-1}\|>0$ for $s\in [0,1]$.
Let $C_1=\sup(\|Q_s\|)$, $C_2=\sup(\|Q_s^{-1}\|)$.
For  $s_0,s\in [0,1]$, 
by Lemma \ref{lm:gap_distance_compare}, we have
\begin{eqnarray*}
\hat\delta(\Graph(M_{s_0}^*A_{s_0}M_{s_0},M_{s}^*A_{s}M_{s})&&=
\hat\delta(Q_{s_0}\Graph(A_{s_0}),Q_s\Graph(A_s))\\
&&\leq C_1\hat\delta(A_{s_0},A_{s})+C_2\|Q_{s}-Q_{s_0}\|.
\end{eqnarray*}
By the continuity of $A_s$ and $Q_s$, we see that 
for any $\epsilon >0$ there is $\delta_1>0$ such that for any $s\in (s_0-\epsilon,s_0-\epsilon)$, we have
$\hat\delta(\Graph(A_{s_0}),\Graph(A_s))<\epsilon/(2C_1)$ and
$\|Q_s-Q_{s_0}\|<\epsilon/(2C_2)$.
Then we have $\hat\delta(\Graph(M_{s_0}^*A_{s_0}M_{s_0},M_{s}^*A_{s}M_{s})<\epsilon.$
This complete the proof.
\end{proof}

Now we give the proof of  Theorem \ref{th:coginva}.
 \begin{proof}[Proof of Theorem \ref{th:coginva}]
 Please note that \eqref{th1f1} is from  Lemma \ref{lem: continuous}. 
 We first prove the case $M_s\equiv M$.
 Let $S=\cal{FS}(\cH_2)$, $K=I$,  $f(A)=M^*AM$.     Please note that
 $$\dim\ker (M^*AM) =\dim (M^{-1}\ker A)=\dim\ker A$$ for each
$A\in \cal{FS}(\cH_2) $. 
Furthermore, we have $\frac{d}{dt} M^*(A+tI)M|=M^*M>0$, 
so $M^*(A+tI)M$ is a positive curve.
Then by Lemma \ref{abstract_decomposition}, we have
$$sf(A_s; s\in[a,b])=sf(M^*A_sM; s\in[a,b])\quad  for \quad  M\in \cal{L}^*(\cH_1,\cH_2).$$
Now we consider the two family $M_{a+t(s-a)}^*A_sM_{a+t(s-a)}$, $(t,s)\in [0,1]\times[a, b]$. By the homotopy invariant property of spectral flow, we have
\begin{equation*}
sf(M_a^*A_sM_a)=sf(M_s^*A_s M_s)-sf(M_{a+t(b-a)}^*A_bM_{a+t(b-a)}).
\end{equation*}
Note that $\dim \ker M_{a+t(b-a)}^*A_bM_{a+t(b-a)}$ is a constant 
 which implies  $sf(M_{a+t(b-a)}^*A_bM_{a+t(b-a)})=0$.
It follows that 
 $$sf(M_a^*A_sM_a)=sf(M_s^*A_s M_s).$$ This complete the proof.
\end{proof}

As an example, we consider the one parameter family of  linear Hamiltonian systems
\begin{eqnarray} \dot{z}(t)= J B_s(t)z(t), (s,t)\in[0,1]\times[0,T],   \label{ham2.1}
  \end{eqnarray}
where  $B(t)\in C([0,1]\times[0,T], \cal{S}(\R^{2n}))$. 
    The boundary condition is given by 
    \bea (x_s(0),x_s(T))\in \Lambda_s \in Lag(4n), \nonumber \eea
    where we assume $\Lambda_s$ is continuous depend on $s$.
  
   Let $A_s=-J\d|_{E(\Lambda_s)}$ which is path of  self adjoint operator on                              
  $\cH:=L^2([0,T],\mathbb C^{2n})$with domain $$E(\Lambda_s)=\{x\in W^{1,2}([0,T], \mathbb C^{2n}), (x(0),x(T))\in\Lambda_s \}.$$      
  We define  $B_s$  by $(B_sx)(t)=B_s(t)x(t)$.  It is well known that 
  $A_s, A_s-B_s\in \cal{FS}(\cH)$ with domain $E_s$.  Let $\ga_s(t)$ be the fundamental solution of \eqref{ham2.1}, i.e. \bea \dot{\ga}_s(t)=JB_s(t)\ga_s(t),\eea
  then $$\ga_s(t)\in\Sp(2n):=\{P\in \cal{L}^*(\R^{2n}), P^*JP=J  \}, $$   which implies $Gr(\ga_s(T))\in Lag(4n ) $. 
   The following formula  which gives the  relation of spectral flow and Maslov index (please refer Theorem \ref{th:6.1}) 
  \bea  -sf(A_s-B_s)=\mu(\Lambda_s, Gr(\gamma_s(T))  ). \nonumber  \eea
 
  Let $P_s(t)\in C^1([0,1]\times[0,T], \Sp(2n))$, then   $P_s\in C^1([0,1], \cal{L}^*(\cH))$, hence $(P^{*}_s)^{-1}(A_s-B_s)P^{-1}_s\in\cal{FS}(\cH)$ with domain
   \bea  P_sE_s= \{x\in W^{1,2}([0,T], \mathbb C^{2n}), (x(0),x(T))\in\hat{P}_s(T)\Lambda_s \} ,    \nonumber  \eea where
 $\hat{P}_s(t)=diag(I_n, P_s(t))$. 
    Direct compute show that \bea (P^{*}_s)^{-1}(-J\d|_{E(\hat{P}_s(T)\Lambda_s)}-B_s)P^{-1}_s=A_s-\hat{B}_s, \nonumber \eea
  where $\hat{B}_s(t)= -J\dot{P}_s(t)P^{-1}_s(t)+(P^{*}_s(t))^{-1}B(t)P^{-1}_s(t) $.    From Theorem \ref{th:coginva}, we have 
  \bea  sf(-J\d|_{E(\hat{P}_s(T)\Lambda_s)}-\hat{B}_s)= sf((P^{*}_s)^{-1}(A_s-B_s)P^{-1}_s) . \label{coninva1} \eea
From \eqref{th:morma}, we can express the left of \eqref{coninva1} as Maslov index. In fact, 
the fundamental solution is $P_s(t)\ga_s(t)$, and the boundary conditions is given by $(\hat{P}_s(T)\Lambda_s $. Hence we have 
 \bea  sf(-J\d|_{E(\hat{P}_s(T)\Lambda_s)}-\hat{B}_s)= \mu(\hat{P}_s(T)\Lambda_s, \hat{P}_s(T)Gr(\gamma_s(T))  ). \nonumber \eea
Formula  \eqref{coninva1}  implies that 
  \bea \mu(\Lambda_s, Gr(\gamma_s(T))  ) =\mu(\hat{P}_s(T)\Lambda_s, \hat{P}_s(T)Gr(\gamma_s(T))  ),  \nonumber \eea
 which is just the  symplectic invariant property (\ref{adp1.4}) of Maslov index.

\section{Decomposition of Spectral flow under cogredient invariant  }

In this section, we will prove the decomposition formula for spectral flow. Suppose $g\in\cal{M}(\cH)$ with
$\sigma(g)=\{\lambda_1,\cdots,\lambda_n\}$, then  
\bea \cH=\sum_{1\leq i\leq n}\cH_{\lambda_i},  \label{3.1a}\eea
where  $\cH_{\lambda_i}:=\ker(g-\lambda_i)^m $ for large enough $m$.
Note that  $(\lambda-\lambda_1)^m$ and $(\lambda-\lambda_2)^m$ are coprime, 
then there are polynomials $p_1,p_2$ such that $p_1(\lambda)(\lambda-\lambda_1)^m+p_2(\lambda)(\lambda-\lambda_2)^m=1$.
For each $x\in \cH_{\lambda_1}\cap \cH_{\lambda_2}$,  we have
$$
x=p_1(g)(g-\lambda_1)^m x+p_2(g)(g-\lambda_2)^m x=0.
$$
Similarly  we have $\cH_{\lambda_i}\cap \cH_{\lambda_j}=0$ with $i\neq j$. So the decomposition \eqref{3.1a}  is a inner direct sum.
\begin{lem}
	$g\in\mathcal{M}(\cH)$  if and only if there exist $\lambda_1,\cdots, \lambda_n\in\mathbb{C}$ such that $\Pi_{i=1}^n (g-\lambda_i)^m=0$.
\end{lem}
\begin{proof}
	We only need to show that $g\in\mathcal{M}(\cH)$  if $\Pi_{i=1}^n (g-\lambda_i)^m=0$. 
	Let $G_l(\lambda) $ be the polynomial $\Pi_{i=1}^{l-1} (\lambda-\lambda_i)^m\Pi_{i=l+1}^n(\lambda-\lambda_i)^m$.
	Then $G_1,G_2,\cdots,G_n$ are coprime polynomials.
	It follows that there are polynomials $a_i(\lambda),(1\le i\le n)$, such that 
	$$
	\sum_{i=1}^n a_i(\lambda)G_i(\lambda)=1.
	$$ 
It follows that $\sum_{i=1}^n a_i(g)G_i(g)=\Id$.
Then we can conclude that 
$$\cH =\sum_{1\leq i\leq n}G_i(g)\cH.$$
We also have $(g-\lambda_i)^mG_i(g)\cH=\Pi_{i=1}^n (g-\lambda_i)^m\cH=0$, which implies \eqref{3.1a}.

\end{proof}

We have  the following lemmas.
\begin{lem} \label{lm:A_orthogonal}
	 Let $A\in\cal{FS}(\cH)$ with domain $E$.  Suppose $g\in\cal{M}(\cH)$, which satisfied  $$g^*Ag=A,\quad gE=E,  $$ then   $\cH_\lambda, \cH_\mu$ is $A$-orthogonal if $\lambda\bar{\mu}\neq1$, i.e. 
\bea  (Ax,y)=0,\quad if \quad x\in \cH_\lambda\cap E,\quad y\in \cH_\mu\cap E. \eea 
\end{lem}
\begin{proof} 
Let $x\in \ker(g-\lambda)^m \cap E$, $y\in  \ker(g-\mu)^n\cap E$ with $m,n\ge 1$.
We see that $(Ax,y)=0$ if $m+n=2$. In fact,  $$
(Ax,y)=(Agx,gy)=\lambda\bar\mu(Ax,y)
$$
implies $(Ax,y)=0$ since $\lambda\bar\mu\neq1$. 
Assume that $(Ax,y)=0$ if $m+n\le k$.
Note that $(g-\lambda)x\in \ker(g-\lambda)^{m-1}\cap E$, $(g-\mu)x\in \ker(g-\mu)^{n-1}\cap E$, $gx\in \ker(g-\lambda)^{m}\cap E$ and $gy\in \ker (g-\mu)^m\cap E$.
If $m+n=k+1$,
We have
$$
(Ax,y)=(Agx,gy)=(A(g-\lambda)x,g y)+(A\lambda x,(g-\mu)y)+\lambda\bar{\mu}(Ax,y)=\lambda\bar\mu(Ax,y).
$$
Since $\lambda\bar\mu\neq 1$, we have $(Ax,y)=0$.
By induction, we have $(Ax,y)=0$ with $x\in \cH_{\lambda}\cap E$ and $y\in \cH_{\mu}\cap E$.  This complete the proof.
\end{proof}

\begin{lem}\label{lm:ker_dom_decom}
	Under the condition of Lemma \ref{lm:A_orthogonal}, then
	$\ker A=\sum_{1\le i\le n}\ker A\cap \cH_i$ and $E=\sum_{1\le i\le n}E\cap \cH_{i}$.
\end{lem}
\begin{proof}
	Note that $E$ is a invariant subspace of $g$.
	Then $\Pi_{1\le i\le n}(g-\lambda_i)^m=0$ on $E$.
	It follows that $E=\sum_{1\le i\le n}\ker (g|_E-\lambda_i)^m=\sum_{1\le i\le n}E\cap \cH_{i}$.
	We have  $$g^*A g(\ker A)=A\ker A=0.$$ It follows that $g(\ker A)\subset \ker A$. So $\ker A$ is a invariant subspace of $g$. Similarly we have
		$\ker A=\sum_{1\le i\le n}\ker A\cap \cH_i$.
	This complete the proof. 
\end{proof}

For $A\in\cal{C}(\cH)$, assume that $\cH=\sum_{1\le i\le k}\cH_i$ where 
all of  $\cH_i$ are closed subspaces of $\cH$. 
Let $E$ be the domain of $A$ and
assume that $E=\sum_{1\le i\le k} E\cap \cH_i$. $\cH_i$, $\cH_j$ are $A$-orthogonal if $i\neq j$. 
Recall that we set \bea F_\lambda=\left\{\begin{array}{cc} \cH_\lambda, \quad  if\quad  \lambda\in \mathbb{U};
\\ \cH_\lambda+\cH_{\bar{\lambda}^{-1}}, \quad \quad  if \quad  \lambda \notin \mathbb{U}.   \end{array}\right.\nonumber \eea
then we have 
 $ \cH=\sum_{1\leq i\leq k}F_{\lambda_i} $ 
and $F_{\lambda_i}, F_{\lambda_j}$ are $A$-orthogonal if $i\neq j$.

Let $X=\bigoplus_{1\le i \le k} F_{\lambda_i}$, 
we define an inner product on $X$:
$$
((x_1,x_2,\cdots,x_k),(y_1,y_2,\cdots,y_k))=\sum_{1\le i\le k} (x_i,y_i),
$$
where $(x_i,y_i)$ is the inner product in $\cH$.
Then $X$ is a Hilbert space and the map  $$ M: (x_1,x_2,\cdots,x_k)\to \sum_{1\le i \le k} x_i$$ is a homeomorphism from $X$ to $\cH$.

Please note  that $A|_{F_{\lambda_i}}$ is the map  $M^*AM:M^{-1}F_{\lambda_i}\to M^{-1}( F_{\lambda_i})$.  It is a self-adjoint Fredholm operator on $M^{-1}F_{\lambda_i}$ with domain $M^{-1}(E\cap F_{\lambda_i})$.  It follows that 
	\begin{equation}
	\ker(A|_{F_{\lambda_i}})=\ker(AM)\cap M^{-1}(F_{\lambda_i})=M^{-1}(\ker A\cap F_{\lambda_i}). \nonumber
	\end{equation}

\begin{prop}\label{pro:decomposition}
	 Suppose $g\in \cal{M}^*(\cH) $,  $A_s\in C([0,1], \cal{FS}(\cH))$ with fixed domain $E$ and $gE=E$.  We assume    $g^* A_sg=A_s$ of $s\in[0,1]$, then we have 
\bea sf(A_s)=sf(A_s|_{F_{\lambda_1}})+\cdots+sf(A_s|_{F_{\lambda_k}}).  \eea 
\end{prop}
\begin{proof} 

By Theorem \ref{th:coginva}, we have $M^*A_sM\in C([0,1],\cal{FS}(X))$,
and 
$$
sf(A_s)=sf(M^*A_sM).
$$
Note that $X=\bigoplus_{1\le i\le k} F_{\lambda_i}$ is an orthogonal decomposition.
 By the Direct sum property of spectral flow,  we have
\begin{equation*}
sf(A_s)=sf(M^*A_sM)=\sum_{1\le i\le k}sf(A_s|_{F_{\lambda_i}}).
\end{equation*} This complete the proof. 
\end{proof}

\begin{lem} \label{lm:sf_notin_circle}
	If $\lambda\notin\mathbb{U}$ then we have 
\bea sf(A_s|_{F_{\lambda}})=\frac{1}{2}(dim\ker (A_1|_{F_{\lambda}}) -dim\ker (A_0|_{F_{\lambda}}) )  .  \eea 

\end{lem}
\begin{proof} 
Recall that $A_s|_{F_\lambda}$ is the operator 
$M^*A_sM: M^{-1}(F_\lambda)\to M^{-1}(F_{\lambda})  $ and 
$M^{-1}(F_{\lambda})=M^{-1}\cH_\lambda+ M^{-1}\cH_{\bar\lambda^{-1}}$.
We also have $M^{-1}\cH_\lambda\perp M^{-1}\cH_{\bar\lambda^{-1}} $.
Let $Q$ be the map $x+y\to -x+y$ with $x\in M^{-1}\cH_\lambda, y\in  M^{-1}\cH_{\bar\lambda^{-1}}$. Then $Q$ is invertible and $Q^*=Q$.
Let $x_1,x_2\in M^{-1}(F_{\lambda}\cap E)$,  $y_1,y_2\in M^{-1}(F_{\bar\lambda^{-1}}\cap E)$.
We have
\begin{equation*}
(QM^*A_sMQ(x_1+y_1),(x_2+y_2))=-(M^*A_sM(x_1+y_1),(x_2+y_2)).
\end{equation*} 
It follows that $-A_s|_{F_\lambda}=Q(A_s|_{F_\lambda})Q$.
Then by Theorem \ref{th:coginva}, we have
$$2 sf(A_s|_{F_{\lambda}})=sf(A_s|_{F_{\lambda}})+sf(QA_s|_{F_\lambda}Q)=sf(A_s)+sf(-A_s)=dim \ker(A_1|_{F_{\lambda}})-dim \ker ( A_0|_{F_{\lambda}}).$$
The lemma then follows.
\end{proof}

\begin{proof}[Proof of Theorem \ref{th:decom}]
By Proposition \ref{pro:decomposition} and Lemma \ref{lm:sf_notin_circle} ,
we only need to show that 
$$
\frac{1}{2}(dim\ker (A_1|_{\hat{F}}) -dim\ker (A_0|_{\hat{F}}))=\sum_{\lambda\notin \mathbb{U}}\frac{1}{2}(dim\ker (A_1|_{F_{\lambda}}) -dim\ker (A_0|_{F_{\lambda}}) ) .
$$
In fact $\ker A_1|_{\hat{F}}=\ker A_1\cap \hat{F}$.
By Lemma \ref{lm:ker_dom_decom}, we see that  $\ker A_1\cap \hat{F}=\sum_{\lambda\notin \mathbb{U}}\ker(A_1)\cap F_{\lambda}$.
It follows that $dim\ker (A_1|_{\hat{F}})=\sum_{\lambda\notin \mathbb{U}}dim \ker(A_1|_{F_{\lambda}})$. It is also true for $A_0$.
The theorem then follows.
 \end{proof}

\begin{cor}\label{cor2.1} Under the condition of Theorem \ref{th:decom}, if $\sigma(M)\cap \mathbb{U}=\emptyset$, then  \bea  sf(A_s)=\frac{1}{2}(dim\ker (A_1) -dim\ker (A_0)).    \eea
If the path is closed,  then $$sf(A_s)=0.$$
\end{cor}

\begin{rem}\label{re3.8}
In the case  $B$ is compact with respect to $A$,  the spectral flow $A-sB$ is only depend on the end points, thus  we define  the relative Morse index by  (follows \cite{ZL99})
 \bea  I(A,A-B)=-sf(A-sB; s\in[0,1]).  \eea 
Especially, when $A$ is positive, then $I(A,A-B)=m^-(A-B)$ is just the Morse index of $A-B$, i.e. the total number of negative eigenvalues. It is obvious that Theorem \ref{th:decom} and Corollary \ref{cor2.1} give the decomposition formula of relative Morse index and Morse index.
\end{rem}

\section{Applications to Hamiltonian systems}
In this section, we will give the applications for Hamiltonian systems. We list $6$ cases which are common in applications.

For  $\Lambda\in Lag(4n)$, we  consider the  solution of  the flowing linear Hamiltonian systems
\begin{eqnarray} \dot{z}(t)= J B(t), \quad (z(0), z(T))\in\Lambda,  \label{3.1}
  \end{eqnarray}
where  $B(t)\in C([0,T], \cal{S}(\R^{2n}))$. 
   Recall that $A=-J\d$  is self adjoint operator on                              
  $\cH:=L^2([0,T],\mathbb C^{2n})$  with domain $$E_\Lambda=\{x\in W^{1,2}([0,T], \mathbb C^{2n}),  (x(0), x(T))\in\Lambda \},$$      
then  $A, A-B\in \cal{FS}(\cH)$.   We will construct $g\in\mathcal{M}(\cH)$ such that \bea g^*Ag=A,\quad g^*Bg=B,\quad  g E_\Lambda=E_\Lambda. \label{4.2} \eea
In order to make $g E_\Lambda=E_\Lambda$,  $g$ is always assumed to preserve the boundary condition, that is $g\Lambda=\Lambda$ which means 
\bea ((gx)(0),(gx)(T))\in \Lambda \quad if \quad  (x(0),x(T))\in\Lambda. \nonumber \eea
   Hence we have \bea g^*(A-sB)g=A-sB, \quad s\in\R.\nonumber \eea 
and get the  decomposition formula \eqref{decomf}.

   It is well known that for $P\in\Sp(2n)$,  if $\lambda\in\sigma(P)$, then $\bar{\lambda},  \lambda^{-1},  \bar{\lambda}^{-1}\in\sigma(P)$ and possess the same geometric and algebraic multiplicities \cite{Lon02}.  Case 1  is given by symplectic matrix.

Case 1.   For $P\in\Sp(2n)$, and satisfied $P\Lambda=\Lambda$  
which means if $(x(0),x(T))\in\Lambda$, then $(Px(0),Px(T))\in\Lambda$. 
 Let    \bea  (gx)(t)=Px(t),\label{g1}  \eea                        
  then it is  obvious that $(g^*x)(t)=P^*x(t)$, $g^*Ag=A$,  $g\Lambda=\Lambda$.  Moreover, we assume $P^*B(t)P=B(t)$, then 
 $ g^*Bg=g$, hence we have \eqref{4.2}. 
 It is obvious that  $g\in\mathcal{M}(\cH)$ and \bea  \sigma(g)=\sigma(P). \nonumber\eea                     
  Let $V_\lambda=\ker(P-\lambda)^{2n}$, then $\cH_\lambda=L^2([0,T], V_\lambda)$.                    
  
    Case 2.  
  For  $S\in\Sp(2n)$, we  consider the $S$-periodic solution of  \eqref{3.1}, that is  \bea z(0)= S z(T), \label{bds} \eea
and moreover  we assume   \bea S^*B(0)S=B(T). \label{Bc}\eea
       We assume  
 \eqref{3.1} with $S$-periodic boundary conditions  admits a $\mathbb Z_k$   symmetry. More precisely,  let $P\in\Sp(2n)$ and $PS=SP$, the group generator $g$ is defined by 
 \bea  (gx)(t)=\left\{\begin{array}{cc} Px(t+\frac{T}{k}), \quad t\in[0,\frac{k-1}{k}T];
\\ PS^{-1}x(t+\frac{T}{k}-T), \quad t\in[\frac{k-1}{n}T,T].   \end{array}\right.  \label{cird} \eea                            
Easy computation show that  $g\in \cal{L}(\cH)$ and $g E=E$. By direct computation, we get the adjoint operator  $g^*$.
\begin{lem}  The adjoint operator $g^*$ is given by \bea  (g^*x)(t)=\left\{\begin{array}{cc} (S^*)^{-1}P^*x(t+T-\frac{T}{k}), \quad t\in[0,\frac{T}{k}];
\\ P^*x(t-\frac{T}{k}), \quad t\in[\frac{T}{k},T].   \end{array}\right.  \eea   
\end{lem}                 
 \begin{proof}

Let $y\in L^2([0,T],\mathbb{C}^{2n})$. We see that 
$$
\int_{0}^{\frac{n-1}{n}T}(Px(t+T/k),y(t))dt= \int_{T/n}^{T}(x(t),P^*y(t-T/k))dt,
$$
and
$$
\int_{\frac{k-1}{k}T}^{T} (PS^{-1}x(t+T/k-T), y(t))dt=\int_{0}^{T/k}(x(t),(S^*)^{-1}P^*y(t+T-T/k)).
$$
Then we have checked  $<gx,y>_{L^2}=<x,g^*y>_{L^2}$ for each $x,y\in L^2([0,T],\mathbb{C}^{2n})$. 
 \end{proof}              
 We assume $B(t)$ satisfied              
       \bea  B(t)=\left\{\begin{array}{cc} (S^*)^{-1}P^*B(t+T-\frac{T}{k})PS^{-1}, \quad t\in[0,\frac{T}{n}];
\\ P^*B(t-\frac{T}{k})P, \quad t\in[\frac{T}{n},T].   \end{array}\right.  \label{BSy1}    \eea     
Please note that \eqref{BSy1} implies \eqref{Bc}, and \eqref{4.2} is satisfied.
            Since  $$ (g^kx)(t)=P^kS^{-1}x(t),  $$
            which is a multipliticity operator on $\cH$.   Then \bea \sigma(g^k)=\sigma(P^kS^{-1}).\nonumber \eea   
 To simply the notation, for $\Omega\in\mathbb C$, we define
 \bea \Omega^{\frac{1}{k}} =\{z\in\mathbb C, z^k\in\Omega   \}.\nonumber\eea
  By this notation, we have $\sigma(g)\in(\sigma(P^kS^{-1}))^{\frac{1}{n}}$.   For $\lambda\in\sigma(g)$, $\cH_\lambda=\ker(g-\lambda)^{2n}$.

  Case 3.  
 We consider the generalized brake symmetry. 
  We call  a matrix $M$  anti-symplectic if it satisfied
\bea  M^*JM=-J.   \eea
We denote by $\Sp_a(2n)$ the set of anti-symplectic matrices. For $M_1,M_2\in \Sp_a(2n)$ and $M_3\in\Sp(2n)$, then it is obvious that 
$$ M_1M_2\in \Sp(2n) ,  \quad  M_1M_3\in\Sp_a(2n) .   $$
We list some basic property of $\Sp_a(2n)$ follows.
\begin{lem}  If $M\in\Sp_a(2n)$, $\lambda\in\sigma(M)$, then $\bar{\lambda},-\lambda^{-1}, -\bar{\lambda}^{-1}\in\sigma(M)$ and posses the same geometric and algebraic multiplicities.
\end{lem}
\begin{proof}
Note that $M^*=-JM^{-1}J^{-1}$. 
Let $\lambda\in \mathbb{C}\backslash\{0\}$.  
It follows that $(M^*-\lambda)=-J(M^{-1}+\lambda)J^{-1}$.
Then we have
$$\dim\ker(M-\bar\lambda)=\dim\ker(M^*-\lambda)=\ker(M^{-1}+\lambda)=\dim\ker(M+\lambda^{-1}).$$
And we also have
$$
\overline{\det(M-\bar\lambda)}=\det(-M^{-1}-\lambda)=
\det(-M^{-1}\lambda)\det(M+\lambda^{-1}).
$$
It follows that 
$\dim\ker(M-\bar\lambda)^{2n}=\dim\ker(M+\lambda^{-1})^{2n}$.
So $\lambda,-\lambda^{-1}\in \sigma(M)$ and posses the same geometric and algebraic multiplicities.
Specially, if $M$ is a real matrix, $\lambda,\bar\lambda,-\lambda^{-1},-\bar\lambda^{-1}\in \sigma(M)$ and posses  the same geometric and algebraic multiplicities. 
\end{proof}

Similar with the symplectic matrix, we have the following results. 
\begin{lem} Let $M\in \Sp_a(2n)$ .  
Let $\lambda,\mu \in \sigma(M)$.
Let $V_\lambda=\ker(M-\lambda)^{2n}$, $V_\mu=\ker(M-\mu)^{2n}$.
Then we have $(Jx,y)=0$ if  $\lambda\bar{\mu}\neq-1$.
\end{lem}
\begin{proof} 
Let $x\in \ker(M-\lambda)^p$,$y\in \ker(g-\mu)^q$ with $p,q\ge 0$.
We see that $(Jx,y)=0$ if $p+q=0$.
Assume that $(Jx,y)=0$ if $p+q\le k$.
Note that $(M-\lambda)x\in \ker(M-\lambda)^{p-1}$ ,$(M-\mu)x\in \ker(M-\mu)^{q-1}$.
If $p+q=k+1$,
We have
$$
(Jx,y)=-(JMx,My)=-(J(M-\lambda)x,M y)-(J\lambda x,(M-\mu)y)-\lambda\bar{\mu}(Jx,y)=-\lambda\bar\mu(Jx,y).
$$
Since $\lambda\bar\mu\neq - 1$, we have $(Ax,y)=0$.
By induction, we have $(Jx,y)=0$ with $x\in \ker(M-\lambda)^{2n}$ and $y\in \ker(M-\mu)^{2n} $.	 This complete the result. 
 \end{proof}
    
   We assume \eqref{3.1}   admits a  generalized brake symmetry.  More exactly, for $N\in\Sp_a(2n)$, let        
     \bea   (gx)(t)=Nx(T-t).  \label{braked}   \eea     
We assume $g\Lambda=\Lambda$, that is \bea (Nx(T),Nx(0))\in\Lambda, \,  if \quad (x(0),x(T))\in\Lambda, \eea        
   then $gE=E$.    Obviously,     
          $(g^*x)(t)=N^*x(T-t)$.  We assume 
          \bea  N^*B(T-t)N=B(t),  \eea
   then  \eqref{4.2} is satisfied. 

    Please note that  for  the $S$-periodic boundary conditions, $NS^{-1}=SN$ implies $g\Lambda=\Lambda$.
   Separated boundary conditions is another kind of important boundary conditions. More preciselly, we consider solution of  (\ref{3.1}) under the boundary conditions\bea x(0)\in V_0,\quad  x(T)\in V_1, \nonumber\eea
   where $V_0,V_1\in Lag(2n)$. In this case $g$ is defined by \eqref{braked}, for $N\in\Sp_a(2n)$ which satisfied \bea NV_0=V_1,\quad  NV_1=V_0,\nonumber \eea 
   then $g\Lambda=\Lambda$.

  Obviously, 
  we have   $$ (g^2x)(t)=N^2x(t),  $$ hence $g\in\mathcal{M}(\cH)$ and 
            \bea \sigma(g)=(\sigma(N^2))^{\frac{1}{2}}. \nonumber \eea
For $\lambda\in\sigma(g)$, $\cH_\lambda=\ker(g-\lambda)^{2n}$.

From Theorem \ref{th:decom} we get the decomposition of spectral flow. Since on the finite interval $B$ is relative compact with respect to $A$,  then from Remark \ref{re3.8}, we have 
\bea  I(A,A-B)=\sum_{i=1}^m I(A|_{F_{\lambda_i}},A|_{F_{\lambda_i}}-B|_{F_{\lambda_i}})+\frac{1}{2}(dim\ker ((A-B)|_{\hat{F}}) -dim\ker (A|_{\hat{F}})). \label{decomrela}  \eea

All the above discussions can be applied  to Sturm-Liouville  systems, so we not give the detail in all cases, instead we only consider the following two cases which have  clearly background.

Case 4. We consider the   one parameter family   Sturm-Liouville  system
\bea   -(G_s(t)\dot{x})'+R_s(t)x(t)=0, \quad  x(0)=Sx(T),\,  \dot{x}(0)=S\dot x(T), s\in[0,1].  \label{sl1}  \eea
       where $S\in\cal{L}^*(\R^n)$.   We suppose $G_s(t),R_s(t)\in \cal{S}(n)$, instead the Legender convex condition we only assume $G_s(t)$ is invertible.   let $P\in\cal{L}^*(\R^n)$ and $PS=SP$, the group generator $g$ is defined as same form of \eqref{cird}.  We assume 
 \bea  G_s(t)=\left\{\begin{array}{cc} (S^*)^{-1}P^*G_s(t+T-\frac{T}{n})PS^{-1}, \quad t\in[0,\frac{T}{n}];
	\\ P^*G_s(t-\frac{T}{n})P, \quad t\in[\frac{T}{n},T].   \end{array}\right.     \eea    
 \bea  R_s(t)=\left\{\begin{array}{cc} (S^*)^{-1}P^*R_s(t+T-\frac{T}{n})PS^{-1}, \quad t\in[0,\frac{T}{n}];
	\\ P^*R_s(t-\frac{T}{n})P, \quad t\in[\frac{T}{n},T].   \end{array}\right.      \eea    
   Then \bea  g^*(-(G_s(t)\frac{d}{dt})'+R_s)g=-(G_s(t)\frac{d}{dt})'+R_s,     \nonumber\eea    
     and we could give the decomposition of spectral flow from Theorem \ref{th:decom}.   
     
     This case include the Bott-type formula of Semi-Riemann manifold \cite{HPY17a}. Let $c$  be a space-like or time-like closed geodesic on $n+1$ dimension Semi-Riemann manifold $(M, \mathfrak{g})$ with period $T$.  We choose a parallel $\mathfrak{g}$-orthonormal frame $e_i(t)$ alone $c$, and satisfied $\mathfrak{g}(e_i(t),\dot{c}(t))=0$. Assume 
     \bea \mathfrak{g}(e_i,e_j)=\left\{\begin{array}{cc} 0,   \quad i\neq j;
	\\ 1,\quad  1\leq i=j\leq n-\nu; \\ -1,\quad  n-\nu\leq i=j\leq n  \end{array}\right.  \nonumber \eea
 and     \bea (e_1(0),\cdots,e_n(0))=(e_1(T),\cdots,e_n(T))P , \nonumber\eea 
     then  $P^TGP=G$ with $G=diag(I_{n-\nu},-I_\nu)$.

     Writing the $\dot{c}$ $\mathfrak{g}$-orthogonal Jacobi vectorfield alone $c$ as  $J(t)=\sum_{i=1}^nu_i(t)e_i(t)$, then  we get the linear second order system of ordinary differential equations
     \bea  -G\ddot{u}+R(t)u(t)=0,\quad  t\in[0,T],  \eea
     where  $R$ is symmetry matrices which is get by the curvature. A period solution is satisfied 
     $$u(0)=Pu(T).$$
For $\omega\in\mathbb{U}$,     let \bea E^2_{\omega,T}:=\{u\in W^{2,2}([0,T],\mathbb{C}^n)|u(0)=\omega Pu(T), \dot{u}(0)=\omega P\dot{u}(T)\},\nonumber \eea           
   then  \bea A^\omega_{s,T}=-G\frac{d^2}{dt^2}+R(t)+sG \nonumber \eea   
 is self-adjoint Fredholm operators on $L^2([0,T],\mathbb{C}^n)$  with domain $E^2_{\omega,T}$.  
             
   It has proved in \cite{HPY17a} that there exist $s_0$ sufficiently large such that for $s\geq s_0$,   $A^\omega_s$ is nondegenerate.         
      The $\omega$ spectral index of $c$ is defined by 
      \bea  i^\omega_{spec}(c) := sf(A^\omega_{s,T}; s\in [0, +\infty)).  \nonumber\eea          
   Let $c^{(m)}$ be the $m$-th iteration of $c$, then     
 \bea  i^\omega_{spec}(c^{(m)}) := sf(A^\omega_{s,mT}; s\in [0, +\infty)).\nonumber  \eea  
   Let $S=P^m$, $G_s=G$, $R_s=R(t)+sG$,  $(gu)(t)=Pu(t+T)$, then from Case 4. we get the decomposition of spectral flow.  Since $g^m=\omega$, then $$ \sigma(g)=\{\omega\}^{\frac{1}{m}}. $$
Let $\omega_j$ be the $m$-th root of $\omega$, then 
\bea \cH_{\omega_j}= \ker(g-\omega_j)=\{ u(t)=\omega_j Pu(t+T) \}. \nonumber  \eea
     We have \bea sf(A^\omega_{s,mT}; s\in [0, +\infty))=\sum_{\omega_j^m=\omega}sf(A^{\omega_j}_{s,T}; s\in [0, +\infty)). \nonumber \eea   
Hence we get the Bott-type iteration formula \cite{HPY17a}
\bea   i^\omega_{spec}(c^{(m)}) =\sum_{\omega_j^m=\omega} i^{\omega_j}_{spec}(c).    \eea
  Obviously, we can consider the case of brake symmetry, since it is similar, we omit the detail.

  Case 5.
 Now we consider the case of heteroclinic orbits, for the one parameter family linear Hamiltonian system 
\bea \dot{x}=JB_s(t)x(t), \quad t\in\R, \quad s\in[0,1]. \label{case5f1} \eea
 Let $B_s(\pm\infty)=\lim_{t\to\pm\infty}B_s(t) $ exist  and satisfied the hyperbolic condition, i.e. $$\sigma(JB_s(\pm\infty))\cap i\R=\emptyset, \quad s\in[0,1].$$
Let $\cH=L^2(\R, \mathbb C^{2n})$,   it is well known that $A-B_s \in\cal{FS}(\cH)$ with domain $E=W^{1,2}(\R, \mathbb{C}^{2n})$.

We assume 
          \bea  N^*B_s(-t)N=B_s(t), \nonumber \eea
   then $g^*B_s g=g$. Easy computation show that $g^*Ag=A$, then we have 
   \bea  g^*(A-B_s)g=A-B_s, \quad s\in[0,1]. \nonumber \eea
Obviously,   we have   $$ (g^2x)(t)=N^2x(t),  $$ hence $g\in\mathcal{M}(\cH)$ and   then \bea \sigma(g)=(\sigma(N^2))^{\frac{1}{2}}. \nonumber \eea
For $\lambda\in\sigma(g)$, $\cH_\lambda=\ker(g-\lambda)^{2n}$, then we get the decomposition formula from Theorem \ref{th:decom}. 

 In the case $N^2=I$, let \bea \cH_\pm=\ker{g\mp I}=\{x\in\cH,Nx(-t)=\pm x(t)\}, \label{hpm}\eea
 we have 
\bea  sf(A-B_s)=sf(A|_{\cH_+}-B_s|_{\cH_+} )+sf(A|_{\cH_-}-B_s|_{\cH_-} ). \label{brakehamil}    \eea

   Now we consider the case of Homoclinics. For the linear Hamiltonian system
   \bea   \dot{x}(t)=JB(t)x(t),\quad t\in\mathbb{R}, \label{homo}\eea
   assume $\lim_{t\to\pm\infty}B(t)=B_*$ and $JB_*$ is hyperbolic.  In this case, $B-B_*$ is relative compact with respect to 
   $A-B_*$, where $A=-J\frac{d}{dt}$.  The relative index is defined by \bea  I(A-B_*, A-B)=-sf(A-B_*+s(B-B_*)).  \nonumber\eea
  In the case \eqref{homo} is a linear system of Homoclinic orbits $z$, the index of $z$ is defined by \cite{CH07} $$i(z)= I(A-B_*, A-B).$$  
 Assume $ N^*B(-t)N=B(t)$ and $N^2=I$,   from \eqref{brakehamil}, we have 
 \bea   I(A-B_*, A-B)= I(A|_{\cH_+}-B_*|_{\cH_+}, A|_{\cH_+}-B|_{\cH_+})+I(A|_{\cH_-}-B_*|_{\cH_-}, A|_{\cH_-}-B|_{\cH_-}).\label{dechomo}\eea

Case 6.    We consider the  one parameter    Sturm-Liouville  system on $\R$
\bea  -(G(t)\dot{x})'+R(t)x(t)=0,\quad  t\in\R  \eea
where $G(t),R(t)\in \cal{S}(n)$.   We assume there exist $\delta>0$, such that  $G(t)>\delta$ for  $t\in\R$,  and there exist $T, \delta_1, \delta_2>0$  such that  
\bea \delta_1<R(t)<\delta_2,\quad for\quad t\geq |T|. \label{slheter} \eea

  The Morse index of $\mathcal{A}:=-(G(t)\frac{d}{dt})'+R(t) $ is defined by the maximum dimension of the subspace such that $\mathcal{A}$ restricted on it is negative definite.
It is well known that $m^-(\mathcal{A})$ is finite under the condition \eqref{slheter}.  Obviously
\bea  m^-(\mathcal{A})=sf(\mathcal{A}+sG(0); s\in[0,+\infty)).\nonumber \eea
We assume there exist  $N\in\Sp_a(n)$, such that 
          \bea  N^*R(-t)N=R(t),\quad  N^*G(-t)N=G(t).\nonumber \eea
    Let \bea  gx(t)=Nx(-t), \nonumber \eea
    then $g^*(\mathcal{A}+sG(0))g=\mathcal{A}+sG(0)$. Since $g^2=N$, then $\sigma(g)=(\sigma(N))^{\frac{1}{2}}$, and we get the decomposition formula of Morse index.
    
    \bea  m^-(\mathcal{A})=\sum_{i=1}^j m^-(\mathcal{A}|_{F_{\lambda_i}})-dim\ker \mathcal{A}|_{\hat{F}}.  \nonumber  \eea
    In the case $N^2=I$, we have 
    \bea  m^-(\mathcal{A})=m^-(\mathcal{A}|_{\cH_+})+m^-(\mathcal{A}|_{\cH_-}) . \eea

\section{Relation with the Maslov index}

In this section, we will give some Bott-type iteration formulas  of Maslov-index. In the what follows $g$ pointed  as the Matrix-like operator appear in Case 2, 3, 5.  To avoid discuss too many technique details, we only consider the case $$g^m=\omega I$$ for some $\omega\in\mathbb{U}$.  Let $\omega_1,\cdots,\omega_m$ be the $m$-th roots of $\omega$, and let $\cH_i=\ker(g-\omega_i)$.

 In Case 2, 3, 5, $\cH(\mathcal{I})=L^2(\mathcal{I}, \mathbb{C}^{2n})$ where $\mathcal{I}$ is some finite interval or $\mathbb{R}$,   and  $E$ is $W^{1,2}(\mathcal{I}, \mathbb{C}^{2n})$ which satisfied some boundary conditions. We choose a subinterval $\hat{\mathcal{I}}\subset \mathcal{I}$, and let $\mathcal{T}$ be the restricition map from $\cH$ to $\cH(\hat{\mathcal{I}}):=L^2(\hat{\mathcal{I}}, \mathbb{C}^{2n})$, that is  
 \bea  (\mathcal{T}f)(x)=f(x),\quad  x\in \hat{\mathcal{I}}. \eea
 $\hat{\mathcal{I}}$ is called a fundamental domain if for any $i=1,\cdots,m$, $\mathcal{T}$ is a bijection from $\cH_i$ to $L^2(\hat{\mathcal{I}}, \mathbb{C}^{2n})$.  Recall that $E_i=E\cap\cH_i$ is domain of $A_s|_{\cH_i}$, then $\mathcal{T}E_i$ is closed in the $W^{1,2}$ norm.  Let $\hat{A}^{i}_s=-J\frac{d}{dt}-B_s$ be the operator on $\cH(\hat{\mathcal{I}})$ with domain $\mathcal{T}E_i$.   
 \begin{lem} Suppose for $s\in[a,b]$,  $\hat{A}^{i}_s$ is   self-adjoint and $\dim\ker(A_s|_{\cH_i})=\dim\ker(\hat{A}^{i}_s)$, then    \bea sf(A_s|_{\cH_i}; s\in[a,b])=sf(\hat{A}^{i}_s; s\in[a,b]). \label{f5.1} \eea  
	\end{lem}
	\begin{proof}
		Note that $(A_s+t\Id)|_{\cH_i}=P_{\cH_i}(A_s+t\Id) P_{\cH_i}=P_{\cH_i}A_s P_{\cH_i}+tP_{\cH_i}$. Since $A_s\in \mathcal{FS}(\cH)$, there is $\epsilon>0$, such that for each $t\in [0,\epsilon]$, $(A_s+t\Id)\in \mathcal{FS}(\cH)$. Then  $(A_s+t\Id)|_{\cH_i}$ is a positive curve on $ \mathcal{FS}(H_i)$ with $t\in [0,\epsilon]$. Note that $-J\frac{d}{dt}-B_s+t\Id$ is also a positive curve on $\mathcal {FS}(\cH(\hat{I}))$,  then \eqref{f5.1}  is from Lemma \ref{abstract_decomposition}. This complete the proof.
	\end{proof} 
 
 Now we consider Case 2.  We assume    $\omega S= P^m$,  then \bea  \cH_i=\ker(g-\omega_i)=\{x\in\cH, \, \omega_i x(t)=Ps(t+\frac{T}{m})  \} .\nonumber \eea
 We choose $\hat{\mathcal{I}}=[0,\frac{T}{m}]$ be the fundamental domain, then \bea \mathcal{T}E_i=\{x\in W^{1,2}([0,\frac{T}{m}],\mathbb{C}^{2n}), \quad \omega_i x(0)=Ps(\frac{T}{m})\}. \nonumber\eea
 From Corollary \ref{6.2}, we have \bea -sf(A_s)=\mu(Gr(\omega S),Gr(\gamma(t));t\in[0,T]),  \quad -sf(\hat{A}^{(i)}_s)=\mu(Gr(\omega_i), Gr(\gamma(t)); t\in[0,T/m]). \nonumber \eea
 Then we have
 \bea   \mu(Gr( S^{-1}),Gr(\gamma(t));t\in[0,T])=\sum_{i=1}^m\mu(Gr(\omega_iP^{-1}), Gr(\gamma(t));t\in[0,T/m]).  \label{bt1}  \eea
 
 \begin{rem}   In the case $P=I_{2n}$,  \eqref{bt1} is the standard Bott-type iteration formula for Hamiltonian systems, please refer \cite{Lon99}, \cite{Lon02} for the detail. In the case $P\in\Sp(2n)\cap \mathbb{O}(2n)$,  \eqref{bt1} is established by Hu and  Sun \cite{HS09},  the general case is proved by   Liu and  Tang \cite{LT15}.
 \end{rem}
 
For  Case 3. We assume $N\in\Sp_a(2n)$  and $N^2=I$. 
Since $N$ is anti-symplectic, we have $(Jx,x)=-(JNx,Nx)=-(Jx,x)=0$ with $x\in \ker(N-I)$.
So $\ker(N-I)$ is a Lagrange subspace of the symplectic space $(\mathbb{R}^{2n},J)$.
 Recall that $(gx)(t)=Nx(T-t)$ and $g\Lambda=\Lambda$. Let \bea \cH_\pm=\ker{g\mp I}=\{x\in\cH,Nx(T-t)=\pm x(t)\},\nonumber \eea
and $\hat{\mathcal{I}}=[0,T/2]$ be the fundamental domain.

For the $S$-periodic boundary conditions, that is  $x(0)=Sx(T)$, then 
\bea \mathcal{T}E_\pm  =\{ x\in W^{1,2}([0,T/2], \mathbb{C}^{2n}),  x(0)\in V^\pm(SN), x(T/2)\in V^\pm(N)   \} .\nonumber \eea
We have 
\bea &&\mu(Gr(S),Gr(\gamma(t)); t\in[0,T])\nonumber \\  &&= \mu(V^+(N), \gamma(t)V^+(SN);t\in[0,T/2] )+\mu(V^-(N), \gamma(t)V^-(SN);t\in[0,T/2] ).   \eea

Similarly  if the boundary condition is given by  $ x(0)\in V_0,\,  x(T)\in V_1$ for 
    $V_0,V_1\in Lag(2n)$ and  $ NV_0=V_1$,   $NV_1=V_0$.  Similar discussion with above, we have 
\bea \mu(V_1, \gamma(t)V_0; t\in[0,T]) = \mu(V^+(N), \gamma(t)V_0;t\in[0,T/2] )+\mu(V^-(N), \gamma(t)V_0;t\in[0,T/2] ).  \label{bt2}  \eea

 \begin{rem}   To  our knowledge, in the case $S=I_{2n}$, $N^2=I$, \eqref{bt2} is first established by Long Zhang and Zhu \cite{LZZ06}. A deep study is given by Liu and Zhang \cite{LZ14a} \cite{LZ14b}. Hu and Sun had established  the case of $S,N\in\mathbb{O}(2n)$, for the case of dihedral group please refer \cite{HPY17b}.
 \end{rem}
 
 Now we consider Case 5.   For $\lambda\in[0,1]$, let $\gamma_\lambda(\tau,t)$ be the fundamental solution of \eqref{case5f1}, that is 
 \bea \dot{\gamma}_\lambda(\tau,t)=JB_\lambda(t)\gamma_\lambda(\tau,t), \quad \gamma_\lambda(\tau,\tau)=I_{2n}.  \eea
 Let \bea V^s_\lambda(\tau):=\{v\in\mathbb{R}^{2n}|\lim_{\tau\to\infty}\gamma_\lambda(\tau,t)v=0 \}, \quad V^u_\lambda(\tau):=\{v\in\mathbb{R}^{2n}|\lim_{\tau\to-\infty}\gamma_\lambda(\tau,t)v=0 \},\nonumber \eea
 be the stable and unstable paths, then $V^s_\lambda(\tau), V^u_\lambda(\tau)\in  Lag(2n)$. 
 
 Recall that in this case, 
 $\cH=L^2(\R,\R^{2n})$ and 
$$
A_\lambda:= -J\frac{d}{dt}-B_s(t) : E=W^{1,2}(\mathbb{R},\mathbb{R}^{2n})\subset \cH\to \cH.
$$
Let $\mathbb{R}^-$ be the fundamental domain, then 
\bea  \mathcal{T}E_\pm=\{x\in W^{1,2}(\mathbb{R}^-,\mathbb{R}^{2n}), x(0)\in V^\pm(N)\}. \nonumber \eea
Let $A^\pm_\lambda$ be the  restricted operators on $\cH(\mathbb{R}^-)$ with domain  $ \mathcal{T}E_\pm$.

From Prop 3.7 of \cite{HP17}, we have 
\bea -sf (A_\lambda; \lambda\in[0,1])= \mu(V^s_\lambda(0), V^u_\lambda(0);\lambda\in[0,1]).   \nonumber    \eea
Similarly 
\bea -sf (A^\pm_\lambda; \lambda\in[0,1])= \mu(V^\pm(N), V^u_\lambda(0);\lambda\in[0,1]).    \nonumber   \eea
Then we have \bea  \mu(V^s_\lambda(0), V^u_\lambda(0);\lambda\in[0,1])=\mu(V^+(N), V^u_\lambda(0);\lambda\in[0,1])+\mu(V^-(N), V^u_\lambda(0);\lambda\in[0,1]). \eea

In the case of homoclinics,  let $A_\lambda=A-B_*-\lambda(B-B_*)$, then  from  \cite{CH07}  or  \cite{HP17}  the index satisfied 
\bea -sf(A_\lambda;\lambda\in[0,1] )=\mu(V^s(+\infty), V^u(t); t\in\mathbb{R})=-\mu(V^s(t), V^u(-t); t\in\mathbb{R}^+  ).   \label{5.12}   \eea
We  have 
\bea -sf(A^\pm_\lambda;\lambda\in[0,1] )=\mu(V^\pm(N), V^u(t); t\in\mathbb{R}^-)=-\mu(V^\pm(N), V^u(-t); t\in\mathbb{R}^+ ).  \label{5.13}    \eea
Compare \eqref{5.12} and \eqref{5.13},  we have \bea  \mu(V^s(+\infty), V^u(t); t\in\mathbb{R})=\mu(V^+(N), V^u(t); t\in\mathbb{R}^-)+\mu(V^-(N), V^u(t); t\in\mathbb{R}^-),\eea or equivalently 
\bea  \mu(V^s(t), V^u(-t); t\in\mathbb{R}^+  )=\mu(V^+(N), V^u(-t); t\in\mathbb{R}^+ )+\mu(V^-(N), V^u(-t); t\in\mathbb{R}^+ ).     \eea

Obviously, we can use $\mathbb{R}^+$ as the fundamental domain, for reader's convenience, we list the formulas below.
Here we let $A_\lambda^{\pm}$ be the restricted operators on $\cH(\mathbb{R}^+)$ with domain $\cT(E_{\pm})$.
\bea -sf (A^\pm_\lambda; \lambda\in[0,1])= \mu(V^s_\lambda(0),V^\pm(N) ;\lambda\in[0,1]).       \eea
	\bea  \mu(V^s_\lambda(0), V^u_\lambda(0);\lambda\in[0,1])=\mu(V^s_\lambda(0),V^+(N) ;\lambda\in[0,1])+\mu(V^s_\lambda(0),V^-(N);\lambda\in[0,1]). \eea
	In  case of the homoclinic, 
	\bea  \mu(V^s(t), V^u(-\infty); t\in\mathbb{R})&=&\mu(V^s(t), V^u(-t); t\in\mathbb{R}^+  )\nonumber \\  &=& \mu( V^s(t),V^+(N); t\in\mathbb{R}^+ )+\mu( V^s(t),V^-(N); t\in\mathbb{R}^+ ).\eea

 In  study the stability problem of homographic solution in planar $n$-body problem, Hu and Ou \cite{HO16} use the McGehee blow up method to get linear heteroclinic system, this system with brake symmetry if the corresponding central configurations  with  brake symmetry.   Please refer \cite{HO16} for the detail.

\section{Appendix: Spectral flow and  Maslov index}

Spectral flow was introduced by Atiyah, Patodi and Singer in their study of index theory on manifold with with boundary \cite{APS76}.
Let $\{A_t ,t\in [0,1]\}$ be a continuous path of self-adjoint Fredholm operators on a Hilbert space $\cH$.
The spectral flow $sf\{A_t\}$ of $A_t$ counts the algebraic multiplicities of the spectral of $A_t$ cross the line
$\lambda=-\epsilon$ with some small positive number $\epsilon$.  For reader's convenience, we list some basic properties of spectral flow.


 {\bf (Stratum homotopy relative to the ends)\/}  For $A_{s,\lambda}\in C([a,b]\times[0,1], \mathcal{FS}(\cH))$, such that  $dim\ker A_{a,\lambda}$ and   $dim\ker A_{b,\lambda}$  is constant, then 
$$sf(A_{s,0};s\in[a,b])=sf(A_{s,1};s\in[a,b]) .$$  

  {\bf (Path additivity)\/} If $A^1,
A^2\in  C\big([a,b];\mathcal{FS}(\cH))\big) $ are  such that 
$ A^1(b)= A^2(a)$, then
$$ 
  sf( A^1_t * A^2_t; t \in [a,b]) = sf( A^1_t; t \in [a,b])+sf( A^2_t; 
t \in [a,b])
$$ 
 where $*$ denotes the usual catenation between the two paths.

 {\bf (Direct sum)\/} If for $i=1,2$, $\cH_i$ are Hilbert space, and  $ A^i\in C\big([a,b];\mathcal{FS}(\cH_i))\big)$, 
   then $$ 
  sf( A^1_t \oplus A^2_t; t \in [a,b]) = sf( A^1_t; t \in [a,b])+sf( A^2_t; 
t \in [a,b]).
$$

  {\bf (Nullity)\/} If $ A\in C\big([a,b];\mathcal{FS}(\cH))\big)$, 
   then $sf( A_t; t \in [a,b])=0$;

   {\bf (Reversal)\/}  Denote the 
same path travelled in the reverse 
direction in  $\mathcal{FS}(\cH)$ by  $\widehat A(t)=A(-t)$. 
Then 
$$ 
sf( A_t; t \in [a,b])=-sf(\widehat A _t; t \in [-b,-a]).
$$ 
\\

The spectral flow is related to Maslov index in Hamiltonian systems. 
We  now briefly reviewing the  Maslov index theory
\cite{Arn67,CLM94, RS93}.    Let $(\mathbb{R}^{2n},\omega)$ be the standard
symplectic space  and $Lag(2n)$ the Lagrangian Grassmanian.
 For two 
continuous paths $L_1(t), L_2(t)$, $t\in[a,b]$ in $Lag(2n)$,
the Maslov index $\mu(L_1, L_2)$ is an integer invariant.   
Here we use the definition from \cite{CLM94}.
We list several properties of the Maslov index. The details could be found
in \cite{CLM94}.

{\bf(Reparametrization invariance)\/}
 Let
$\phi:[c,d]\rightarrow [a,b]$ be a continuous and piecewise smooth
function with $\phi(c)=a$, $\phi(d)=b$, then \bea \mu(L_1(t),
L_2(t))=\mu(L_1(\phi(\tau)), L_2(\phi(\tau))). \lb{adp1.1} \eea

{\bf(Homotopy invariant with end points)\/}
  For two continuous
family of Lagrangian path $L_1(s,t)$, $L_2(s,t)$, $0\leq s\leq 1$, $a\leq
t\leq b$, and satisfies  $dim L_1(s,a)\cap L_2(s,a)$ and $dim
L_1(s,b)\cap L_2(s,b)$ is constant, then  \bea \mu(L_1(0,t),
L_2(0,t))=\mu(L_1(1,t),L_2(1,t)). \lb{adp1.2} \eea

{\bf(Path additivity)\/}
 If $a<c<b$, then then 
\bea \mu(L_1(t),L_2(t))=\mu(L_1(t),L_2(t)|_{[a,c]})+\mu(L_1(t),L_2(t)|_{[c,b]}).
\lb{adp1.3} \eea

{\bf(Symplectic invariance)\/}
  Let $\ga(t)$, $t\in[a,b]$ is a
continuous path in $\Sp(2n)$, then \bea \mu(L_1(t),L_2(t))=\mu(\ga(t)L_1(t), \ga(t)L_2(t)). \lb{adp1.4} \eea

{\bf(Symplectic additivity)\/}
  Let $W_i$, $i=1,2$ be symplectic space, 
$L_1,L_2\in C([a,b], Lag(W_1))$ and $\hat{L}_1,\hat{L}_2\in C([a,b], Lag(W_2))$, 
then \bea \mu(L_1(t)\oplus \hat{L}_1(t),L_2(t)\oplus \hat{L}_2(t))= \mu(L_1(t),L_2(t))+
\mu(\hat{L}_1(t),\hat{L}_2(t)). 
\lb{adp1.5add} \eea

The next Theorem give the relation of spectral flow and Maslov index.
 \begin{thm}\label{th:6.1}
  \bea  -sf(A_s-B_s)=\mu(\Lambda_s, Gr(\gamma_s(T))  ) \label{th:morma}  \eea
  \end{thm}
  The above Theorem is well known \cite{HP17},  we like to give a direct proof here.
    \begin{proof}
    	We use the idea of Lemma \ref{abstract_decomposition}.
    	We only need to prove the theorem locally.
    Let $s_0\in [0,1]$. $A_{s_0}-B_{s_0}+rI,r\in [0,1]$ is a positive path in $\cal{FS}(\cH)$.
    There is $\epsilon>0$ such that  $\ker(A_{s_0}-B_{s_0}+\epsilon I)={0}$.
    Then there is $\delta>0$ such that $\ker(A_{s}-B_{s}+\epsilon I)={0}$, $\forall s\in [s_0-\delta,s_0+\delta]$.
    Without loss of generality, we can assume that   $$\ker(A_{s}-B_{s}+I)={0},\,  \forall s\in [0,1].$$
    Let $\gamma_{s,r}(t)$ be the fundamental solution of the equation
    \begin{eqnarray} \dot{z}(t)= J (B_s(t)-rI)z(t), (s,t)\in[0,1]\times[0,T],\, r\in [0,1].  
    \end{eqnarray}
    Recall that $\dim(\Graph(\gamma_{s,r}(T))\cap \Lambda_s)=\dim\ker(A_s-B_s+rI)$, then
    we have $\mu(\Lambda_s,Gr(\gamma_{s,1}(T)))=0$.

    Then use the homotopy invariant property of spectral flow and Maslov index, we have
    \bea
    \begin{cases}\label{eq:sf_mas}
    	sf(A_s-B_s,s\in [0,1])=sf(A_{0}-B_{0}+rI)-sf(A_1-B_1+rI)\\
    	\mu(\Lambda_s,\Graph(\gamma_s(T)))=\mu(\Lambda_0,\Graph(\gamma_{0,r}(T)))-\mu(\Lambda_1,\Graph(\gamma_{0,r}(T)))
    \end{cases}.
  \eea
Note that $A_0-B_0+rI$ is a positive path in $\cal{FS}(\cH)$, 
then we have 
$$
sf(A_0-B_0+rI)=\sum_{0<r\le 1} \dim\ker(A_0-B_0+rI).
$$
Let $Q_t((x,\gamma(t) x),(y,\gamma(t) y))=<-J\gamma(t)^{-1}\dot{\gamma}(t)x,y>$
which  is a quadratic form on $\Graph(\gamma(t))$. 
Recall that the crossing form of the Lagrangian pair $(\Lambda,\Graph(\gamma(t)))$ is given by
$Q_t|_{Gr(\gamma(t))\cap \Lambda}$.

Note that
 \begin{eqnarray*}
 \frac{\partial}{\partial t}(\gamma_s(t)^{-1}\frac{\partial}{\partial s}\gamma_s(t))=-\gamma_s(t)^{-1}\frac{\partial}{\partial t}\gamma_s(t)\gamma_s(t)^{-1}\frac{\partial}{\partial s}\gamma_s(t)
 +\gamma_s(t)^{-1}\frac{\partial}{\partial s}(\frac{\partial}{\partial t}\gamma_s(t))\\
 =-\gamma_s(t)^{-1}\frac{\partial}{\partial t}\gamma_s(t)\gamma_s(t)^{-1}\frac{\partial}{\partial s}\gamma_s(t)
 +\gamma_s(t)^{-1}\frac{\partial}{\partial s}(JB_s)\gamma_s(t)+\gamma_s(t)^{-1}JB_s\frac{\partial}{\partial s}\gamma_s(t)\\
 =\gamma_s(t)^{-1}\frac{\partial}{\partial s}(JB_s)\gamma_s(t).
 \end{eqnarray*}
Then we have
$\frac{\partial}{\partial t}(-J\gamma_{0,r}(t)^{-1}\frac{\partial}{\partial r}\gamma_{0,r}(t))=-I$, and
it follows that
$-J\gamma_{0,r}(t)^{-1}\frac{\partial}{\partial r}\gamma_{0,r}(t)=-TI$.
Thus  we have
\begin{eqnarray*}
\mu(\Lambda_0,\Graph(\gamma_{0,r}(T)))&&=-\sum_{0<r\le 1}\dim(\Lambda_0\cap \Graph(\gamma_{0,r}(T)))\\
&&=-\sum_{0<r\le 1}\dim(\ker(A_0-B_0+rI))=-sf(A_0-B_0+rI).
\end{eqnarray*}
Similarly 
$$
\mu(\Lambda_1,\Graph(\gamma_{1,r}(T)))=-sf(A_1-B_1+rI).
$$
Then by \eqref{eq:sf_mas}, we get \eqref{th:morma}.
  \end{proof}
  From  the homotopy invariance of Maslov index, we have
  \begin{cor} \label{6.2}  \bea  -sf(A-sB)=\mu(\Lambda,Gr(\gamma(t)), t\in[0,T]). \eea 
  \end{cor}


\noindent {\bf Acknowledgements.} We would like to thank Professor Alessandro Portaluri some helpful discussion with us for the spectral flow.

\end{document}